\documentclass[final,3p,times]{elsarticle}
\journal{Discrete Mathematics}
\biboptions{sort&compress}

\usepackage[T1]{fontenc}
\usepackage{amsmath,amssymb,mathtools}
\usepackage{amsthm}
\usepackage{graphicx}
\usepackage{booktabs,array,tabularx}
\usepackage{enumitem}
\usepackage{float}
\usepackage[section]{placeins}

\usepackage{hyperref}
\usepackage{xcolor}
\hypersetup{
    colorlinks=true,
    linkcolor=blue,
    urlcolor=blue,
    citecolor=blue
}
\usepackage{thmtools,thm-restate}

\newtheorem{lemma}{Lemma}
\newtheorem{corollary}{Corollary}
\newtheorem{subclaim}{Claim}

\newenvironment{claimproof}
{\par\noindent\emph{Proof of the claim.}\ }
{\hfill $\triangleleft$\par}

\DeclareMathOperator{\rad}{rad}
\DeclareMathOperator{\diam}{diam}
\DeclareMathOperator{\MP}{mp}

\newcommand{\ceil}[1]{\left\lceil #1\right\rceil}
\newcommand{\floor}[1]{\left\lfloor #1\right\rfloor}
\newcommand{\ball}[2]{N_{#2}[#1]}

\begin{document}
\begin{frontmatter}

\title{Broadcast Domination Number is at Most Twice the Multipacking Number\tnoteref{label1}}

\tnotetext[label1]{A preliminary version of this paper, entitled
``Broadcast Domination Number is at Most Twice the Multipacking Number,''
has been accepted at the {\it 19th International Conference on
Combinatorial Optimization and Applications} (COCOA 2026)
(\url{https://www.cocoa2026.com/accepted-papers.html}).}

\author[isi]{Sk Samim Islam\corref{cor1}}
\ead{samimislam08@gmail.com}
\cortext[cor1]{Corresponding author}

\affiliation[isi]{organization={Indian Statistical Institute},
                  city={Kolkata},
                  country={India}}


\begin{abstract}
  For a connected graph $G=(V,E)$ with at least two vertices, a function
	$ f : V \rightarrow \{0, 1, 2, \dots ,$ $ \diam(G)\} $ is called a \emph{broadcast} on $ G $. For each
	vertex $ u \in V  $, if  there exists a vertex $ v $ in $ G $ (possibly, $ u = v $) such that $ f (v) > 0 $ and
	$ d(u, v) \leq f (v) $, then $ f $ is called a dominating broadcast on $ G $.  The cost of a dominating broadcast $f$ is the quantity $ \sum_{v\in V}f(v) $. The minimum cost of a
	dominating broadcast is the broadcast domination number of $G$, denoted by $ \gamma_b(G) $. 
	    
	    A multipacking is a set $ M \subseteq V  $ in a
	graph $ G = (V, E) $ such that for every vertex $ v \in V $ and for every integer $ r \geq 1 $, the
	ball of radius $ r $ around $ v $ contains at most $ r $ vertices of $ M $, that is, there are at most
	$ r $ vertices in $ M $ at a distance at most $ r $ from $ v $ in $ G $. The
	multipacking number of $ G $ is the maximum cardinality of a multipacking of $ G $ and
	is denoted by $ \MP(G) $.
    
It is known that $\MP(G)\leq\gamma_b(G)$. In 2014, Hartnell and Mynhardt
proved that $\gamma_b(G)\leq3\MP(G)-2$ whenever $\MP(G)\geq2$. In 2019,
Beaudou, Brewster, and Foucaud improved this bound to
$\gamma_b(G)\leq2\MP(G)+3$ and conjectured that
$\gamma_b(G)\leq2\MP(G)$. We solve
their conjecture by proving that $\gamma_b(G)\leq2\MP(G)$  for every graph $G$. Our proof is
constructive and yields a linear-time $2$-approximation algorithm for the
Maximum Multipacking problem, improving the earlier approximation factor
$2+o(1)$.
\end{abstract}


\end{frontmatter}

\section{Introduction}\label{sec:introduction}

Covering and packing are fundamental problems in graph theory and algorithms~\cite{cornuejols2001combinatorial}.  We study two dual covering and packing problems called \emph{broadcast domination} and \emph{multipacking}. The broadcast domination problem is motivated by telecommunication networks. Imagine a network with radio towers that can transmit information within a certain radius 
 $r$ for a cost of $r$. The goal is to cover the entire network while minimizing the total cost. The multipacking problem is its natural packing counterpart and generalizes various other standard packing problems. Unlike many standard packing and covering problems, these two problems involve arbitrary distances in graphs, which makes them challenging. The goal of this paper is to study the relation between these two parameters in general graphs.

For a connected graph $G=(V,E)$, $d_G(u,v)$ is the length of a shortest path joining two vertices $u$ and $v$ in  $G$, and we simply write $d(u,v)$ when there is no confusion. Let $N_r[v]:=\{u\in V:d(v,u)\leq r\}$, i.e. a ball of radius $r$ around $v$.  The \textit{eccentricity} $e(w)$  of a vertex $w$ is $\min \{r:N_r[w]=V\}$. The \textit{radius} of the graph $G$ is $\min\{e(w):w\in V\}$, denoted by $\rad(G)$.   The \textit{diameter} of the graph $G$ is $\max\{d(u,v):u,v\in V\}$, denoted by $\mathrm{diam}(G)$. A shortest path of length $\diam(G)$ is called a \textit{diametral path} of $G$.

The covering problem we study is broadcast domination. For a connected graph $ G = (V, E) $ with at least two vertices and diameter $\diam(G)$, a function
	$ f : V \rightarrow \{0, 1, 2, \dots , \diam(G)\} $ is called a \textit{broadcast} on $ G $. Suppose $G$ is a graph with a broadcast $f$. For each
	vertex $ u \in V  $, if  there exists a vertex $ v $ in $ G $ (possibly, $ u = v $) such that $ f (v) > 0 $ and
	$ d(u, v) \leq f (v) $, then $ f $ is called a \textit{dominating broadcast} on $ G $.
Its \textit{cost} is $\sigma(f)=  \sum_{v\in V}f(v)$. The minimum cost of a dominating broadcast in $G$ (taken over all dominating broadcasts)  is the \textit{broadcast domination number} of G, denoted by $ \gamma_{b}(G) $.  Thus, $ \gamma_{b}(G) = \min_{f\in D(G)} \sigma(f)$, where $D(G)$ is the set of all dominating broadcasts on $G$. For the one-vertex graph $K_1$, we allow a broadcast of strength $1$, so $\gamma_b(K_1)=1$. For disconnected graphs, broadcasts are defined independently on the connected components. For the empty graph, the broadcast domination number is $0$.

An \textit{optimal broadcast} or \textit{minimum dominating broadcast} on a graph $G$ is a dominating broadcast with a cost equal to $ \gamma_{b}(G) $.	  Suppose $V(G)=\{v_1,v_2,v_3,\dots,$ $v_n\}$. Let $c$ and $x$ be the vectors indexed by $(i,k)$ where $v_i\in V(G)$ and $1\leq k\leq \diam(G)$, with the  entries $c_{i,k}=k$ and $x_{i,k}=1$ when $f(v_i)=k$ and $x_{i,k}=0$ when $f(v_i)\neq k$. Let $A=[a_{j,(i,k)}]$ be a matrix with the entries 
\begin{center}	$a_{j,(i,k)}=
    \begin{cases}
        1 & \text{if }  v_j\in N_k[v_i]\\
        0 & \text{otherwise. }
    \end{cases} $
 \end{center}

	Hence, the broadcast domination number can be expressed as an integer linear program:  
 \begin{center}
	    $\gamma_b(G)=\min \{c\cdot x :  Ax\geq \mathbf{1}, x_{i,k}\in \{0,1\}\}.	$
     \end{center}

At an optimal solution, at most one radius is chosen at each vertex: if several radii are selected at one vertex, keeping only the largest preserves coverage and reduces the cost.

 Replacing the integrality conditions in the covering program by $x\geq0$ gives a linear program whose dual is $\max\{\mathbf{1}\cdot y:yA\leq c,\ y\geq0\}$. Restricting $y$ to binary vectors gives the \textit{Maximum Multipacking} problem. In 2013, Brewster, Mynhardt, and Teshima \cite{brewster2013new} formally introduced the notion of a multipacking.    A \textit{multipacking} is a set $ M \subseteq V  $ in a
	graph $ G = (V, E) $ such that   $|N_r[v]\cap M|\leq r$ for each vertex $ v \in V(G) $ and for every integer $  r\geq 1 $. The \textit{multipacking number} of $ G $ is the maximum cardinality of a multipacking of $ G $ and it
	is denoted by $ \MP(G) $.   A \textit{maximum multipacking} is a multipacking $M$  of a graph $ G  $ such that	$|M|=\MP(G)$. If $M$ is a multipacking, we define   a vector $y$ with the entries $y_j=1$ when $v_j\in M$ and $y_j=0$ when $v_j\notin M$.  So,  \begin{center}
	    $\MP(G)=\max \{y\cdot\mathbf{1} :  yA\leq c, y_{j}\in \{0,1\}\}.$ 
\end{center}

\medskip
\noindent\textbf{Brief Survey:}
Broadcast domination was introduced by Erwin, while multipacking was
introduced by Teshima as its natural packing counterpart; see
\cite{erwin2004dominating,brewster2013new,beaudou2019broadcast}.  Hartnell and
Mynhardt proved that $\gamma_b(G)\leq3\MP(G)-2$ whenever
$\MP(G)\geq2$ \cite{hartnell2014difference}.  Beaudou, Brewster, and
Foucaud later improved this to $\gamma_b(G)\leq2\MP(G)+3$ and conjectured
that $\gamma_b(G)\leq2\MP(G)$ for every graph
\cite{beaudou2019broadcast}.  The coefficient $2$ is best possible because a recent work gives a family of hypercubes for which
$\gamma_b(G)/\MP(G)$ tends to $2$ for arbitrarily large $\MP(G)$
\cite{rajendraprasad2025multipacking}.  From an algorithmic viewpoint, Minimum Dominating Broadcast can be solved in polynomial-time~\cite{heggernes2006optimal}. Whereas, 
 Multipacking problem is NP-complete~\cite{das2025hardnessMultipackingESA} and there is a $(2+o(1))$-approximation algorithm for general graphs~\cite{beaudou2019broadcast}. Polynomial-time algorithms are known for trees and, more generally, for strongly chordal graphs~\cite{brewster2019broadcast}. Approximation algorithms with ratio $(\tfrac{3}{2}+o(1))$ have been obtained for chordal graphs~\cite{das2023relation,das2026relationDAM} and cactus graphs~\cite{das2025multipacking}.

\medskip
\noindent\textbf{Our Contribution:}
First, we improve the previously known bound
$\MP(G)\geq(\rad(G)-3)/2$ of Beaudou, Brewster, and
Foucaud~\cite{beaudou2019broadcast}.

\begin{restatable}{theorem}{radiusmultipackingbound}
\label{thm:radius-multipacking}
For every connected graph $G$,
$\MP(G)\geq\floor{\rad(G)/2}$.
\end{restatable}

To prove this result, we use a two-path structure similar to the one
used by Beaudou, Brewster, and Foucaud~\cite{beaudou2019broadcast}.
We first choose an isometric path of length $\rad(G)$ and then
construct another isometric path of the same length starting from
a vertex near the middle of the first path. We select every third
vertex along suitable portions of the two paths. By analyzing
the distances between the selected vertices, we show that they form
a multipacking. In fact, our construction gives a multipacking of
size at least $\ceil{\rad(G)/2}$ whenever
$\rad(G)\not\equiv5\pmod6$.

Next, we prove the conjecture of Beaudou, Brewster, and
Foucaud~\cite{beaudou2019broadcast}, improving their bound
$\gamma_b(G)\leq2\MP(G)+3$.

\begin{restatable}{theorem}{mainbroadcastbound}
\label{thm:main-broadcast-bound}
For every graph $G$, $\gamma_b(G)\leq2\MP(G)$.
\end{restatable}

To prove this result, we use the multipacking constructed in the
proof of Theorem~\ref{thm:radius-multipacking}. As mentioned above, our construction
gives a multipacking of size at least $\ceil{\rad(G)/2}$
whenever $\rad(G)\not\equiv5\pmod6$. For a connected graph with at least two vertices,
$\gamma_b(G)\leq\rad(G)$. Hence, this construction already gives
$\gamma_b(G)\leq2\MP(G)$ whenever
$\rad(G)\not\equiv5\pmod6$. Even when
$\rad(G)\equiv5\pmod6$, the desired bound follows if
$\gamma_b(G)<\rad(G)$. Thus, it remains to consider the case
$\rad(G)\equiv5\pmod6$ and $\gamma_b(G)=\rad(G)$.
For this case, we use a ball-cover argument. We choose two balls
whose radii sum to $\rad(G)-1$. Since
$\gamma_b(G)=\rad(G)$, these balls cannot cover the entire graph.
We choose a vertex outside their union and show that it can be added
to a multipacking obtained from the selected paths. This gives a
larger multipacking and proves the desired bound. The one-vertex
case is immediate, and the result extends to disconnected graphs
by applying the connected case to each component.

Finally, we provide a linear-time $2$-approximation algorithm
for Maximum Multipacking, improving the earlier approximation
factor $2+o(1)$~\cite{beaudou2019broadcast}.
Our algorithm is based on the two-path construction used to
prove Theorem~\ref{thm:main-broadcast-bound}.
However, computing the radius by running breadth-first search
from every vertex would take $O(nm)$ time. To avoid this,
we use two breadth-first searches to construct the required
isometric paths. The smaller of the two computed eccentricities
gives an upper bound $R$ on $\MP(G)$. We shorten and reindex
the first path so that our construction can be applied with
this value of $R$. When necessary, we check whether two balls
cover the graph and choose an uncovered vertex if they do not.
This requires at most two additional breadth-first searches.
We show that the resulting set is a multipacking of size at
least $\MP(G)/2$. Thus, the algorithm runs in $O(n+m)$ time.

\begin{restatable}{theorem}{multipackingapproximation}
\label{thm:multipacking-approximation}
There is an $O(n+m)$-time $2$-approximation algorithm for Maximum
Multipacking on a graph $G$ given by adjacency lists, where
$n=|V(G)|$ and $m=|E(G)|$.
\end{restatable}

\medskip
\noindent\textbf{Conference Version:}
A preliminary version of this paper, entitled
``Broadcast Domination Number is at Most Twice the Multipacking Number'',
has been accepted at the International Conference on Combinatorial
Optimization and Applications (COCOA 2026). The preliminary version
presented an $O(nm)$-time $2$-approximation algorithm for Maximum
Multipacking, where $n$ and $m$ denote the numbers of vertices and
edges of the input graph, respectively. The present version improves
the running time to $O(n+m)$.

\medskip
\noindent\textbf{Organisation:}
In Section~\ref{sec:preliminaries}, we present the definitions and basic
results used throughout the paper. In Section~\ref{sec:radius-bound}, we
establish the relation between the multipacking number and the radius. In
Section~\ref{sec:broadcast-multipacking}, we prove the main bound relating
broadcast domination and multipacking. In Section~\ref{sec:approximation},
we present the resulting approximation algorithm for Maximum Multipacking.
Finally, we conclude in Section~\ref{sec:conclusion}.

\section{Preliminaries}\label{sec:preliminaries}

All graphs are finite, simple, and undirected.  Unless stated otherwise, they
are connected and have at least two vertices.   A path
$P=(v_0,v_1,\ldots,v_\ell)$ is called \emph{isometric} if
$d(v_i,v_j)=|i-j|$ for all $0\leq i,j\leq\ell$.  Equivalently, an isometric path is a
shortest path between its end vertices.    

It is known that for every connected graph $G$, $\MP(G)\leq \gamma_b(G)\leq\rad(G)$~\cite{brewster2013new,erwin2001cost,erwin2004dominating}. The following useful criterion is due to Beaudou, Brewster, and Foucaud
\cite{beaudou2019broadcast}.

\begin{lemma}[\cite{beaudou2019broadcast}]\label{lem:criterion}
Let $M\subseteq V(G)$.  If, for every subset $U\subseteq M$ with $|U|\geq2$, there exist two vertices $x,y\in U$ satisfying $d(x,y)\geq2|U|-1$,
then $M$ is a multipacking of $G$.
\end{lemma}

This can be proved by contradiction.  If $M$ is not a multipacking, some ball
$\ball{z}{r}$ contains a set $U\subseteq M$ of at least $r+1$ vertices.  Every
pair of vertices in $U$ is at distance at most $2r\leq2|U|-2$, so $U$ fails to satisfy
the stated property.

\section{Relation Between Multipacking Number and Radius
(Proof of Theorem~\ref{thm:radius-multipacking})}
\label{sec:radius-bound}

\begin{figure}[t]
    \centering
    \includegraphics[width=.75\textwidth]
    {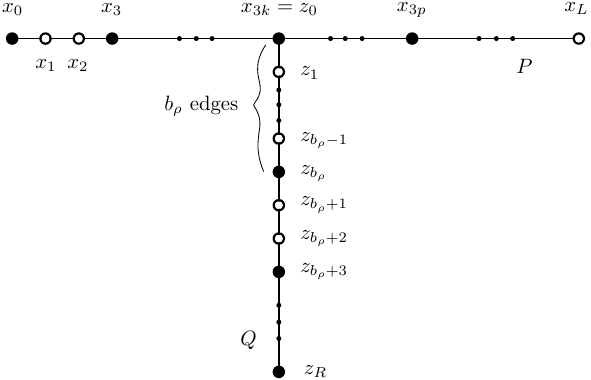}
    \caption{The horizontal path is an isometric path $P=(x_0,\ldots,x_L)$, and
    $Q=(z_0,\ldots,z_R)$ is an isometric path with initial vertex
    $z_0=x_{3k}$. The filled vertices represent the vertices selected for $M_R$.}
    \label{fig:Tstructure}
\end{figure}

Let $P=(x_0,x_1,\ldots,x_L)$ and $Q=(z_0,z_1,\ldots,z_R)$ be
isometric paths in a graph $G$ such that $z_0=x_h$; see Figure~\ref{fig:Tstructure}. Then, for
all indices $i$ and $j$,
\begin{equation}\label{eq:cross-basic}
    d(x_i,z_j)\geq j-|i-h|.
\end{equation}

This is true because, by the triangle inequality, we have
$d(x_h,z_j)\leq d(x_h,x_i)+d(x_i,z_j)$, where
$d(x_h,z_j)=d(z_0,z_j)=j$ and
$d(x_h,x_i)=|i-h|$.

The first path need only contain the vertices up to index
$3\floor{R/3}$. This length condition will allow us to use the same
construction in the linear-time algorithm.

\begin{lemma}\label{lem:Tstructure}
Let $R\geq1$, 
$k=\floor{R/6}$, $p=\floor{R/3}$, and
$\ell=\floor{(R-1)/6}$.  Let $G$ be a graph that contains isometric paths
$P=(x_0,\ldots,x_L)$ and $Q=(z_0,\ldots,z_R)$, where
$L\geq3\floor{R/3}$ and
$z_0=x_{3k}$.  Then $$ M_R=\{x_{3i}:0\leq i\leq p\}
     \cup\{z_{R-3j}:0\leq j<\ell\}$$ is a multipacking in $G$.  Moreover, $$
 |M_R|=
 \begin{cases}
  \ceil{R/2},&R\not\equiv5\pmod6,\\
  (R-1)/2,&R\equiv5\pmod6.
 \end{cases}
$$

\end{lemma}

\begin{proof}
Let $R=6k+\rho$, where $\rho\in\{0,1,2,3,4,5\}$, and  $h=3k$.
Let $A=\{x_{3i}:0\leq i\leq p\}$ and
$B=\{z_{R-3j}:0\leq j<\ell\}$,
so $M_R=A\cup B$.  The largest distance along $P$ from a vertex of $A$
to $x_h$ is
 $$D_\rho=\max\{|i-h|:x_i\in A\}=
 \begin{cases}
  3k,&\rho\in\{0,1,2\},\\
  3k+3,&\rho\in\{3,4,5\}.
 \end{cases}$$

When $B\neq\emptyset$, its smallest index is
$$ b_\rho=\min\{j:z_j\in B\}=
 \begin{cases}
  3k+6,&\rho=0,\\
  3k+\rho+3,&\rho\in\{1,2,3,4,5\}.
 \end{cases}$$

Therefore, $b_\rho>D_\rho$ for each $\rho\in\{0,1,2,3,4,5\}$, when $B\neq\emptyset$.  


\begin{subclaim}
    $A\cap B=\emptyset$.

\end{subclaim}

\begin{claimproof}
     If $B=\emptyset$, then $A\cap B=\emptyset$ trivially. Otherwise, 
    suppose that $x_i=z_j\in A\cap B$. Since $P$ and $Q$ are isometric
and $x_h=z_0$, we have $j=d(z_0,z_j)=d(x_h,x_i)=|i-h|\leq D_\rho$. On the other hand, $z_j\in B$ implies
$j\geq b_\rho>D_\rho$, a contradiction. Hence
$A\cap B=\emptyset$.
\end{claimproof}
\medskip

Now $p=2k$ for $\rho\in\{0,1,2\}$ and $p=2k+1$ for
$\rho\in\{3,4,5\}$, while $\ell=k-1$ for $\rho=0$ and $\ell=k$ otherwise. Therefore,  $$
 |A|=p+1=
 \begin{cases}
  2k+1,&\rho\in\{0,1,2\},\\
  2k+2,&\rho\in\{3,4,5\}.
 \end{cases}
$$
and 

 $$
 |B|=\ell=
 \begin{cases}
  k-1,&\rho=0\\
  k,&\rho\in\{1,2,3,4,5\}.
 \end{cases}
$$
Consequently,
$$
 |M_R|=|A|+|B|=
 \begin{cases}
  3k,&\rho=0,\\
  3k+1,&\rho\in\{1,2\},\\
  3k+2,&\rho\in\{3,4,5\},
 \end{cases}
$$
which implies 

$$
 |M_R|=
 \begin{cases}
  \ceil{R/2},&R\not\equiv5\pmod6,\\
  (R-1)/2,&R\equiv5\pmod6.
 \end{cases}
$$

We prove that $M_R$ is a multipacking of $G$ using Lemma~\ref{lem:criterion}.  Let $U\subseteq M_R$ with
$|U|=t\geq2$.  If $U\subseteq A$ or $U\subseteq B$, then the vertices of $U$
occur at indices differing by multiples of three on an isometric path.  Consequently, the vertices of $U$ having minimum and maximum
indices are at distance at least $3(t-1)\geq2t-1$.

Assume henceforth that $\alpha=|U\cap A|\geq1$,
$\beta=|U\cap B|\geq1$, and $t=\alpha+\beta$.  Let
$s=\min\{|i-h|:x_i\in U\cap A\}$.
Choose $i^*$ such that $x_{i^*}\in U\cap A$ and $|i^*-h|=s$, and let
$j^*=\max\{j:z_j\in U\cap B\}$.
Since the indices of the vertices of $B$ form an arithmetic progression with
common difference three and $|U\cap B|=\beta$, we have
$j^*\geq b_\rho+3(\beta-1)$.  By \eqref{eq:cross-basic},
\begin{equation}\label{eq:cross-lower}
        d(x_{i^*},z_{j^*})
        \geq b_\rho+3(\beta-1)-s.
\end{equation}
Suppose, for a contradiction, we assume that every pair of vertices in $U$ is at distance
at most $2t-2$.  Comparing this upper bound with
\eqref{eq:cross-lower} gives
\begin{equation}\label{eq:s-lower}
        s\geq b_\rho+\beta-2\alpha-1.
\end{equation}

Let $I_A=\{i:x_i\in U\cap A\}$.
Assume first that either $I_A\subseteq\{0,\ldots,h\}$ or
$I_A\subseteq\{h,\ldots,L\}$.
Within either of the two index intervals, consecutive indices corresponding
to vertices of $A$ differ by three.  Hence, among any $\alpha$ such indices,
the minimum value of $|i-h|$ is at most $D_\rho-3(\alpha-1)$.
Combining this inequality with \eqref{eq:s-lower} yields
$\alpha+\beta\leq D_\rho+4-b_\rho$.  For
$\rho=0,1,\ldots,5$, the quantity on the right is, respectively,
$-2,0,-1,1,0,-1$.  This contradicts $\alpha+\beta\geq2$.

It follows that there exist indices $i_-,i_+\in I_A$ satisfying
$i_-<h<i_+$.  Let $i_{\min}=\min I_A$ and $i_{\max}=\max I_A$.

\begin{subclaim}\label{clm:robust-index-span}
 $i_{\max}-i_{\min}\geq2s+3(\alpha-2).$
\end{subclaim}

\begin{claimproof} Let $\alpha_-$ and $\alpha_+$ denote the numbers of indices in $I_A$ smaller
and larger than $h$, respectively. Note that both $\alpha_-$ and $\alpha_+$ are positive.  If $h\notin I_A$, then
$\alpha_-+\alpha_+=\alpha$.    Therefore
$
        i_{\max}-i_{\min}
        \geq \bigl(s+3(\alpha_--1)\bigr)
              +\bigl(s+3(\alpha_+-1)\bigr)
        =2s+3(\alpha-2).
$
If $h\in I_A$, then $s=0$ and
$\alpha_-+\alpha_+=\alpha-1$.  Since both $\alpha_-$ and $\alpha_+$ are
positive,
$
        i_{\max}-i_{\min}
        \geq3\alpha_-+3\alpha_+
        =3(\alpha-1)
        \geq2s+3(\alpha-2).
$
\end{claimproof}
\medskip

Since $P$ is isometric, we have
$d(x_{i_{\min}},x_{i_{\max}})=i_{\max}-i_{\min}$. Recall, we assumed that every pair of vertices in $U$ is at distance
at most $2t-2$.
This fact, Claim~\ref{clm:robust-index-span} and \eqref{eq:s-lower}  imply $
  2(\alpha+\beta)-2
  \geq2s+3(\alpha-2)
  \geq2(b_\rho+\beta-2\alpha-1)+3\alpha-6,
$
which simplifies to
\begin{equation}\label{eq:alpha-lower}
        3\alpha\geq2b_\rho-6.
\end{equation}

For $\rho\in\{0,2,4,5\}$, \eqref{eq:alpha-lower} gives $
 \alpha\geq2k+2,
  \alpha\geq2k+2,
 \alpha\geq2k+3$, and 
 $\alpha\geq2k+4
$ respectively. Each bound exceeds $|A|$, so this is a contradiction.

Suppose $\rho=1$.  Then \eqref{eq:alpha-lower} implies
$\alpha=|A|=2k+1$, so $A\subseteq U$.  In particular,
$x_0,x_{6k}\in U$, and hence $
 6k=d(x_0,x_{6k})
 \leq2|U|-2=2(2k+1+\beta)-2.
$ Thus $\beta\geq k$.  Since $|B|=k$, we have $U=M_R$.  Moreover,
$z_0=x_{3k}\in A$ and $z_R\in B$, so
$d(z_0,z_R)=R=6k+1=2|M_R|-1=2|U|-1$,
which contradicts the assumption on $U$.

Suppose $\rho=3$.  Then \eqref{eq:alpha-lower} implies
$\alpha=|A|=2k+2$.  Hence $A\subseteq U$, and the two vertices
$x_0,x_{6k+3}\in U$ satisfy $d(x_0,x_{6k+3})=6k+3$.  Under the
contradiction assumption, $6k+3\leq2|U|-2=2(2k+2+\beta)-2$,
which implies $\beta\geq k+1>|B|$, a contradiction.

Every case is contradictory.  Hence, by Lemma~\ref{lem:criterion}, $M_R$ is
a multipacking of $G$.
\end{proof}

\begin{lemma}\label{lem:radius-reduction}
Let $G$ be a connected graph of radius $r\geq1$.  Then
$$
 \MP(G)\geq
 \begin{cases}
  \ceil{r/2},&r\not\equiv5\pmod6,\\
  (r-1)/2,&r\equiv5\pmod6.
 \end{cases}
$$
\end{lemma}

\begin{proof}
As $\diam(G)\geq r$, choose an isometric path $P=(x_0,\ldots,x_r)$ of
length $r$.  Let $k=\floor{r/6}$ and $c=x_{3k}$.
Every vertex has eccentricity at least $r$, so there is a vertex $w$ with
$d(c,w)\geq r$.  Take a subpath $Q$ of length $r$ of a shortest path from
$c$ to $w$. Let 
$Q=(z_0,\ldots,z_r)$ with $z_0=c$. $Q$ is also an isometric path.  Apply Lemma~\ref{lem:Tstructure} with
$R=L=r$.
\end{proof}

Theorem~\ref{thm:radius-multipacking} is immediate for $G=K_1$, since
$\MP(K_1)=1$ and $\rad(K_1)=0$. If $G$ has at least two vertices,
then $\rad(G)\geq1$, and Lemma~\ref{lem:radius-reduction} gives
$\MP(G)\geq\floor{\rad(G)/2}$. This proves Theorem~\ref{thm:radius-multipacking}.

\begin{corollary}\label{cor:plus-one}
For every connected graph $G$, $\gamma_b(G)\leq2\MP(G)+1$.
Moreover, the stronger inequality
$\gamma_b(G)\leq2\MP(G)$ holds unless
$\rad(G)\equiv5\pmod6$ and $\gamma_b(G)=\rad(G)$.
\end{corollary}

\begin{proof}
The result is immediate for $G=K_1$, so assume that $G$ has at least two
vertices.  Let $r=\rad(G)$.   Since $\ceil{r/2}\geq (r-1)/2$, we have $\MP(G)\geq \frac{r-1}{2}\geq \frac{\gamma_b(G)-1}{2}$ by Lemma~\ref{lem:radius-reduction}. If $r\not\equiv5\pmod6$, then $\MP(G)\geq \ceil{r/2}\geq r/2\geq \gamma_b(G)/2$ by Lemma~\ref{lem:radius-reduction}. If  $\gamma_b(G)\leq r-1$, then by Theorem~\ref{thm:radius-multipacking}, $\MP(G)\geq \floor{r/2}\geq \frac{r-1}{2}\geq \frac{\gamma_b(G)}{2}$.
\end{proof}


\section{Relation between Broadcast Domination and Multipacking Number
(Proof of Theorem~\ref{thm:main-broadcast-bound})}
\label{sec:broadcast-multipacking}

We now consider the only case left by Corollary~\ref{cor:plus-one}, namely,
$\rad(G)\equiv5\pmod6$ and $\gamma_b(G)=\rad(G)$.
We first prove two statements about the paths themselves. Neither
statement requires $R$ to be the radius of $G$, so both will also
apply to the algorithm in Section~\ref{sec:approximation}.

Let $R=6k+5$, where $k\geq1$. Suppose that $G$ contains isometric
paths $P=(x_0,x_1,\ldots,x_L)$ and $Q=(z_0,z_1,\ldots,z_R)$,
where $L\geq6k+3$ and $z_0=x_{3k}$.
Since $R=6k+5$, we have $\floor{R/3}=2k+1$ and
$\floor{(R-1)/6}=k$. Let 
$X_i=x_{3i}$ for $0\leq i\leq2k+1$ and
$Z_j=z_{R-3j}$ for $0\leq j\leq k-1$.  Then the set $M_R$ from
Lemma~\ref{lem:Tstructure} can be written as $$ M_R=\{X_i:0\leq i\leq2k+1\}
          \cup\{Z_j:0\leq j\leq k-1\}.$$
By Lemma~\ref{lem:Tstructure}, $M_R$ is a multipacking of size $3k+2$.
For all indices $i,j$,
\begin{equation}\label{eq:MR-cross}
        d(X_i,Z_j)
        \geq R-3j-3|i-k|.
\end{equation}
This is true because, by the triangle inequality, we have
$d(X_k,Z_j)\leq d(X_k,X_i)+d(X_i,Z_j)$.

\begin{figure}[t]
    \centering
    \includegraphics[width=.77\textwidth]
    {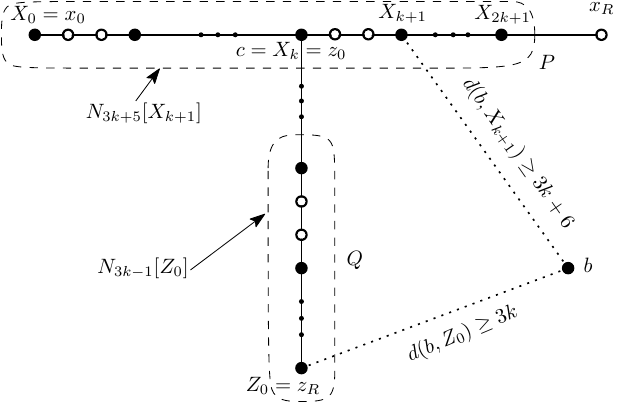}
    \caption{The two-path construction when $L=R$. The horizontal path is an isometric path $P=(x_0,\ldots,x_R)$, and
    $Q=(z_0,\ldots,z_R)$ is an isometric path with initial vertex
    $z_0=x_{3k}$. The filled vertices on $P$ and $Q$ represent the vertices of $M_R$. The vertex $b$ is later added to $M_R$.}
    \label{fig:TstructureOneVertex}
\end{figure}

We first prove a stronger distance property of $M_R$.

\begin{lemma}\label{lem:robust-MR}

Let $R=6k+5$, where $k\geq1$, and let $M_R$ be the set defined above.
Then, every subset $T\subseteq M_R$ with $|T|=t\geq4$ contains two vertices whose
distance is at least $2t+1$.
\end{lemma}

\begin{proof}
Let
$A=\{X_i:0\leq i\leq2k+1\}$ and
$B=\{Z_j:0\leq j\leq k-1\}$. So, $M_R=A\cup B$.
Let $T\subseteq M_R$ with $|T|=t$.

If $T\subseteq A$ or $T\subseteq B$, then the vertices of $T$ corresponding
to the minimum and maximum indices are at distance at least $3(t-1)$.
Since $t\geq4$, we have $3(t-1)\geq2t+1$.

Assume that $T$ intersects both $A$ and $B$. Let
$\alpha=|T\cap A|$, $\beta=|T\cap B|$, and
$t=\alpha+\beta$. Define
$I=\{i:X_i\in T\}$ and $J=\{j:Z_j\in T\}$, and let
$s=\min_{i\in I}|i-k|$.

Let $j_0=\min J$. Since $J$ contains $\beta$ distinct indices from
$\{0,1,\ldots,k-1\}$, we have $j_0\leq k-\beta$. Choose
$i_0\in I$ such that $|i_0-k|=s$. By \eqref{eq:MR-cross},
\begin{equation}\label{eq:robust-cross}
d(X_{i_0},Z_{j_0})
\geq R-3j_0-3|i_0-k|
\geq3k+5+3\beta-3s.
\end{equation}

Suppose, for a contradiction, we assume that every pair of vertices in $T$ is at
distance at most $2t$. Inequality~\eqref{eq:robust-cross} gives
\begin{equation}\label{eq:robust-strong}
3s\geq3k+\beta+5-2\alpha.
\end{equation}

Assume first that either
$I\subseteq\{0,\ldots,k\}$ or
$I\subseteq\{k,\ldots,2k+1\}$.
In the first case, $s\leq k+1-\alpha$, while in the second case,
$s\leq k+2-\alpha$. Thus, in either case,
$s\leq k+2-\alpha$. Combining this inequality with
\eqref{eq:robust-strong} gives
$3k+\beta+5-2\alpha\leq3k+6-3\alpha$, and hence
$\alpha+\beta\leq1$. This contradicts
$\alpha,\beta\geq1$.

It follows that there exist $i_-,i_+\in I$ such that
$i_-<k<i_+$. Let $i_{\min}=\min I$ and $i_{\max}=\max I$.

\begin{subclaim}\label{clm:i_gap}
    $i_{\max}-i_{\min}\geq2s+\alpha-2$.
\end{subclaim}

\begin{claimproof}
    Let $\alpha_-$ and $\alpha_+$ be the numbers of indices in $I$ smaller and
larger than $k$, respectively. If $k\notin I$, then
$\alpha_-+\alpha_+=\alpha$, and hence
$i_{\max}-i_{\min}\geq
(s+\alpha_--1)+(s+\alpha_+-1)=2s+\alpha-2$.
If $k\in I$, then $s=0$ and
$\alpha_-+\alpha_+=\alpha-1$. Therefore,
$i_{\max}-i_{\min}\geq\alpha_-+\alpha_+=\alpha-1
\geq2s+\alpha-2$.
\end{claimproof}
\medskip

Since $P$ is isometric,
$d(X_{i_{\min}},X_{i_{\max}})
=3(i_{\max}-i_{\min})$.
Under the contradiction assumption and Claim~\ref{clm:i_gap}, we obtain
$3(2s+\alpha-2)\leq2t$, and consequently
$6s+\alpha\leq2\beta+6$.
Combining this with \eqref{eq:robust-strong} gives
$3\alpha\geq6k+4$. Since $\alpha\leq2k+2$, we must have
$\alpha=2k+2$. Hence
$I=\{0,1,\ldots,2k+1\}$ and $s=0$.

Substituting $\alpha=2k+2$ and $s=0$ into
\eqref{eq:robust-strong} gives $\beta\leq k-1$. On the other hand,
$X_0,X_{2k+1}\in T$, and therefore
$d(X_0,X_{2k+1})=6k+3$. The contradiction assumption gives
$6k+3\leq2t=2(2k+2+\beta)$, which implies $\beta\geq k$.
This contradicts $\beta\leq k-1$.

Therefore, $T$ contains two vertices whose distance is at least $2t+1$.
\end{proof}

\begin{lemma}\label{lem:augmentation}
Let $R=6k+5$, where $k\geq1$, and let $P,Q$, and $M_R$ be as above.
If
$$
 b\notin\ball{x_{3k+3}}{3k+5}\cup\ball{z_R}{3k-1},
$$
then $M_R\cup\{b\}$ is a multipacking of size $3k+3$.
\end{lemma}

\begin{proof}
Since $b$ lies outside the two balls, we have
\begin{equation}\label{eq:b-main}
        d(b,x_{3k+3})\geq3k+6
        \quad\text{and}\quad
        d(b,z_R)\geq3k.
\end{equation}
Every $X_i$ is within distance $3k+3$ of $x_{3k+3}$, and every $Z_j$ is within
distance $3k-3$ of $z_R$.  Hence $b\notin M_R$, and
$M_R\cup\{b\}$ has size $3k+3$. Let $M=M_R\cup\{b\}$; see Figure~\ref{fig:TstructureOneVertex}.

We prove that $M$ is a multipacking of $G$.  Let $\ball{v}{r}$ be any ball
with integer radius $r\geq1$.  If $b\notin\ball{v}{r}$, then
$|\ball{v}{r}\cap M|\leq r$ because $M_R$ is a multipacking by Lemma~\ref{lem:Tstructure}.  Suppose that
$b\in\ball{v}{r}$, and let $T=\ball{v}{r}\cap M_R$ and $t=|T|$.  We have
$t\leq r$, since $M_R$ is a multipacking.  There is nothing to prove when $t\leq r-1$, so assume $t=r$.

If $r\geq4$, Lemma~\ref{lem:robust-MR} gives two vertices of $T$ at distance at
least $2r+1$, impossible inside a radius-$r$ ball.

Suppose that $r\in\{1,2,3\}$. By the triangle inequality, we have $d(b,X_i)\geq d(b,X_{k+1})-d(X_i,X_{k+1})$ and $d(b,Z_j)\geq d(b,Z_0)-d(Z_j,Z_0)$. Therefore, from \eqref{eq:b-main},
\begin{align}
 d(b,X_i)&\geq 3k+6-3|i-(k+1)|,
        &&0\leq i\leq2k+1, \label{eq:bX}\\
 d(b,Z_j)&\geq3k-3j,
        &&0\leq j\leq k-1. \label{eq:bZ}
\end{align}
Both lower bounds are at least $3$.  Thus a radius-$1$ ball containing $b$
contains no vertex of $M_R$.

For $r=2$, every vertex of $T$ is within distance $4$ of $b$.  Inequalities 
\eqref{eq:bX} and \eqref{eq:bZ} show that the only possibilities are $X_0$ and
$Z_{k-1}$.  But $d(X_0,Z_{k-1})\geq8$
by \eqref{eq:MR-cross}, whereas two vertices in a radius-$2$ ball are at
distance at most $4$.

For $r=3$, every vertex of $T$ is within distance $6$ of $b$.  Inequalities 
\eqref{eq:bX} and \eqref{eq:bZ} show that the only possibilities are $
 X_0, X_1, X_{2k+1}, Z_{k-2},Z_{k-1}$. When $k=1$, the vertex $Z_{k-2}$ is not defined, so the possible
vertices are $X_0,X_1,X_3$, and $Z_0$. Every three of the four
remaining vertices contain one of the pairs
$\{X_0,X_3\}$, $\{X_0,Z_0\}$, or $\{X_1,Z_0\}$, whose distances are at
least $9$, $8$, and $11$, respectively. The first
bound follows because $P$ is isometric and
$d(X_0,X_{3})=9$, while the remaining bounds follow from
\eqref{eq:MR-cross}.  If $k\geq2$, every three vertices
from $
 X_0, X_1, X_{2k+1}, Z_{k-2},Z_{k-1}$ contain one of the 
pairs
$\{X_0,X_{2k+1}\}$, $\{X_0,Z_{k-1}\}$, $\{X_0,Z_{k-2}\}$,
$\{X_1,Z_{k-1}\}$, or $\{X_{2k+1},Z_{k-2}\}$. The distances of these
pairs are at least $15$, $8$, $11$, $11$, and $8$, respectively. The first
bound follows because $P$ is isometric and
$d(X_0,X_{2k+1})=6k+3\geq 15$, while the remaining bounds follow from
\eqref{eq:MR-cross}. 
Thus the three vertices of $T$ contain a pair at distance at least $7$, which
is impossible in a radius-$3$ ball.

Consequently, $M_R\cup\{b\}$ is a multipacking of size $3k+3$.
\end{proof}

\begin{lemma}\label{lem:five-class}
Let $G$ be a connected graph with $\rad(G)\equiv5\pmod6$ and
$\gamma_b(G)=\rad(G)$. Then $\gamma_b(G)\leq2\MP(G)$.
\end{lemma}

\begin{proof}
Write $R=\rad(G)=\gamma_b(G)=6k+5$.

First suppose $k=0$. Choose vertices $u,v$ at distance $5$.
Such vertices exist because $\diam(G)\geq\rad(G)=5$.
The balls $\ball{u}{2}$ and $\ball{v}{2}$ cannot cover $G$, since
otherwise they would define a dominating broadcast of cost $4$.
Choose a vertex $b$ outside their union. Then $d(u,v)=5$ and
$d(b,u),d(b,v)\geq3$. Thus $\{u,v,b\}$ is a multipacking:
a radius-$1$ ball contains at most one of its vertices, a radius-$2$
ball cannot contain both $u$ and $v$, and larger radii impose no
further restriction. Hence $\MP(G)\geq3$ and
$\gamma_b(G)=5\leq2\MP(G)$.

Now suppose $k\geq1$. Construct isometric paths $P$ and $Q$ of
length $R$, with $z_0=x_{3k}$, as in the proof of
Lemma~\ref{lem:radius-reduction}. These paths satisfy the
conditions at the beginning of this section. Consider the balls
$$
 \ball{x_{3k+3}}{3k+5}
 \quad\text{and}\quad
 \ball{z_R}{3k-1}.
$$
Their radii sum to $6k+4=R-1$. Since $\gamma_b(G)=R$, the two balls
do not cover $G$. Choose a vertex $b$ outside their union.
By Lemma~\ref{lem:augmentation}, $M_R\cup\{b\}$ is a multipacking
of size $3k+3$. Therefore,
$$
 \gamma_b(G)=6k+5<2(3k+3)\leq2\MP(G).
$$
\end{proof}

Thus Corollary~\ref{cor:plus-one} and Lemma~\ref{lem:five-class} yield  that, for connected graphs, $\gamma_b(G)\leq2\MP(G)$.

For a disconnected graph $G$ with components
$G_1,\ldots,G_q$, both parameters are additive, i.e., $\gamma_b(G)=\sum_{i=1}^{q}\gamma_b(G_i)
$ and $
\MP(G)=\sum_{i=1}^{q}\MP(G_i)$. Applying the connected case to each component gives
$\gamma_b(G)\leq2\MP(G)$.

 This proves Theorem~\ref{thm:main-broadcast-bound}.

\section{Approximation Algorithm
(Proof of Theorem~\ref{thm:multipacking-approximation})}
\label{sec:approximation}

We avoid computing either $\rad(G)$ or $\gamma_b(G)$.
The following lemma constructs the required paths by two
breadth-first searches. Recall that $e(v)$ denotes the eccentricity
of a vertex $v$.

\begin{lemma}\label{lem:bfs-paths}
Given a connected graph $G$ with at least two vertices, one can
compute an integer $R\geq1$ and isometric paths
$P=(x_0,\ldots,x_L)$ and $Q=(z_0,\ldots,z_R)$ such that
$$
 \gamma_b(G)\leq R,\qquad
 L=3\floor{R/3},\qquad
 z_0=x_{3\floor{R/6}}.
$$
The construction takes $O(n+m)$ time and uses two breadth-first
searches.
\end{lemma}

\begin{proof}
Choose any vertex $s$ and perform a breadth-first search from $s$.
Let $D=e(s)$, and use the search tree to obtain an isometric path
$$
 (v_0,v_1,\ldots,v_D),\qquad v_0=s,
$$
from $s$ to a vertex at distance $D$.
Set $h=3\floor{D/6}$ and $c=v_h$.
Perform a second breadth-first search from $c$, and let $E=e(c)$.
Set
$$
 R=\min\{D,E\},\qquad
 k=\floor{R/6},\qquad
 L=3\floor{R/3},\qquad
 a=h-3k.
$$
Since $R\leq D$, we have $a\geq0$. We also have
$$
 L-3k
 =3\floor{(R+3)/6}
 \leq3\floor{(D+3)/6}.
$$
Consequently,
$$
 a+L
 =h+(L-3k)
 \leq3\floor{D/6}+3\floor{(D+3)/6}
 =3\floor{D/3}
 \leq D.
$$
Thus the subpath
$$
 P=(x_0,\ldots,x_L),\qquad x_i=v_{a+i}\quad(0\leq i\leq L),
$$
is well defined and isometric. Moreover,
$x_{3k}=v_{a+3k}=v_h=c$.

The second search gives an isometric path of length $E$ starting
at $c$. Since $R\leq E$, take its first $R$ edges to obtain
$Q=(z_0,\ldots,z_R)$ with $z_0=c=x_{3k}$.
Finally, a broadcast of strength $D$ at $s$ dominates $G$, and a
broadcast of strength $E$ at $c$ also dominates $G$. Hence
$\gamma_b(G)\leq\min\{D,E\}=R$.

Each breadth-first search takes $O(n+m)$ time. The two paths are
obtained from the parent pointers of the search trees in $O(n)$
time. Computing their subpaths and the indices above also takes
$O(n)$ time.
\end{proof}

We now give the algorithm and prove its approximation guarantee.

\begin{proof}[Proof of Theorem~\ref{thm:multipacking-approximation}]
First suppose $G$ is connected. If $G=K_1$, return its unique
vertex. Otherwise, construct $R,P$, and $Q$ by
Lemma~\ref{lem:bfs-paths}, and put $k=\floor{R/6}$.
We have
\begin{equation}\label{eq:linear-upper}
 \MP(G)\leq\gamma_b(G)\leq R.
\end{equation}
The paths satisfy the conditions of Lemma~\ref{lem:Tstructure}.
Construct the multipacking
$$
 M_R=
 \{x_{3i}:0\leq i\leq\floor{R/3}\}
 \cup
 \{z_{R-3j}:0\leq j<\floor{(R-1)/6}\}.
$$
The algorithm returns a set $M$ as follows.

\medskip
\noindent\textbf{Case 1: $R\not\equiv5\pmod6$.}
Return $M=M_R$.
By Lemma~\ref{lem:Tstructure}, $M$ is a multipacking and
$|M|=\ceil{R/2}$. Therefore, \eqref{eq:linear-upper} gives
$|M|\geq\MP(G)/2$.

\medskip
\noindent\textbf{Case 2: $R=5$.}
Let $u=z_0$ and $v=z_5$, so $d(u,v)=5$.
Test whether $\ball{u}{2}\cup\ball{v}{2}$ covers $G$.
If it does, return $M=\{u,v\}$. The set $M$ is a multipacking,
and the two balls give a dominating broadcast of cost $4$.
Consequently,
$$
 \MP(G)\leq\gamma_b(G)\leq4=2|M|.
$$
If the two balls do not cover $G$, choose a vertex $b$ outside
their union and return $M=\{u,v,b\}$.
As in the proof of Lemma~\ref{lem:five-class}, this is a multipacking
because $d(u,v)=5$ and $d(b,u),d(b,v)\geq3$.
Since $\MP(G)\leq R=5$, we again have $|M|\geq\MP(G)/2$.

\medskip
\noindent\textbf{Case 3: $R=6k+5$ with $k\geq1$.}
Test whether
$$
 \ball{x_{3k+3}}{3k+5}\cup\ball{z_R}{3k-1}
$$
covers $G$. The vertex $x_{3k+3}$ exists because
$L=6k+3\geq3k+3$. If the balls cover $G$, return $M=M_R$.
By Lemma~\ref{lem:Tstructure}, $M$ is a multipacking of size $3k+2$.
The balls give a dominating broadcast of cost $6k+4$, so
$$
 \MP(G)\leq\gamma_b(G)\leq6k+4=2|M|.
$$
If the balls do not cover $G$, choose a vertex $b$ outside their
union and return $M=M_R\cup\{b\}$.
Lemma~\ref{lem:augmentation} applies to these paths and shows that
$M$ is a multipacking of size $3k+3$. By
\eqref{eq:linear-upper},
$$
 \MP(G)\leq R=6k+5<2(3k+3)=2|M|.
$$
This proves the approximation guarantee in every case.

\medskip
\noindent\textbf{Running time.}
Lemma~\ref{lem:bfs-paths} uses two breadth-first searches.
Constructing $M_R$ takes $O(n)$ time by scanning the appropriate
indices of the two paths. In Cases~2 and~3, perform a breadth-first
search from each of the two ball centers. The resulting distance
arrays allow us to test whether the balls cover $G$ and, if not,
find a vertex outside their union by scanning all vertices once.
Thus at most four breadth-first searches are performed, and the
total running time is $O(n+m)$. Apart from the adjacency-list input,
the searches, paths, and output require $O(n)$ space.

For a disconnected graph, the first breadth-first search in each
component also identifies that component. Apply the construction
separately to each component. Let $G_1,\ldots,G_q$
be the components, and let $M_i$ be the returned multipacking of
$G_i$. Each ball lies in a single component, so
$M=\bigcup_{i=1}^qM_i$ is a multipacking of $G$. Moreover,
$$
 |M|=\sum_{i=1}^q|M_i|
 \geq\frac12\sum_{i=1}^q\MP(G_i)
 =\frac{\MP(G)}2.
$$
The component computations use
$O\bigl(\sum_i(|V(G_i)|+|E(G_i)|)\bigr)=O(n+m)$ time in total.
The empty graph is handled by returning the empty set.
\end{proof}

The algorithm can also return a dominating broadcast $f$ with
$\sigma(f)\leq2|M|$. When the two tested balls cover the graph, use
their radii as the broadcast strengths. Otherwise, use a single
broadcast of strength $R$ at the vertex whose eccentricity equals
$R$ in Lemma~\ref{lem:bfs-paths}. Use strength $1$ for an isolated
vertex and combine the broadcasts over components. Thus the
algorithm provides the explicit bounds
$$
 |M|\leq\MP(G)\leq\gamma_b(G)\leq\sigma(f)\leq2|M|.
$$

\section{Conclusion}\label{sec:conclusion}

The factor $2$ in Theorem~\ref{thm:main-broadcast-bound} cannot be reduced. Equality
already occurs for $C_4$ and $C_5$, where $\gamma_b(G)=2$ and $\MP(G)=1$. For more examples, see~\cite{teshima2012broadcasts}. More
significantly, hypercubes form an infinite family of connected graphs for which the ratio 
$\gamma_b(G)/\MP(G)$ tends to $2$ for arbitrarily large $\MP(G)$ \cite{rajendraprasad2025multipacking}.

We give an $O(n+m)$-time $2$-approximation algorithm for Maximum
Multipacking. The algorithm also constructs a dominating broadcast whose
cost is at most twice the size of the returned multipacking.

Although the factor $2$ in the relation between broadcast
domination and multipacking is best possible, this does
not imply that the approximation factor $2$ is tight.
A natural further direction is to determine whether
Maximum Multipacking admits a polynomial-time
approximation algorithm with factor strictly below $2$,
or whether such an improvement is impossible under
standard complexity assumptions.

\bibliographystyle{elsarticle-num}
\bibliography{ref}

\end{document}